\documentclass[a4paper,11pt,reqno]{amsart}

\usepackage{setspace}
\usepackage{geometry} 
\usepackage{amsmath, amsthm, amssymb}
\usepackage{mathtools}      
\usepackage{mathabx}        
\usepackage[bb=fourier,cal=euler,scr=rsfs]{mathalfa}	
\usepackage[shortlabels]{enumitem}       

\usepackage[utf8]{inputenc}

\usepackage{tikz-cd}

\usepackage{comment}

\usepackage{multicol}

\usepackage{transparent}

\newtheorem*{thm*}{Theorem}

\newtheorem{thm}{Theorem}[section]

\newtheorem{quest}[thm]{Question}
\newtheorem{problem}[thm]{Problem}
\newtheorem{cor}[thm]{Corollary}

\newtheorem{lemma}[thm]{Lemma}
\newtheorem{prop}[thm]{Proposition}

\theoremstyle{definition}
\newtheorem{definition}[thm]{Definition}

\newtheorem{remark}[thm]{Remark}
\newtheorem{notation}[thm]{Notations}

\newtheorem{example}[thm]{Example}

\newtheorem*{quest*}{Question}

\newtheorem*{claim1}{Claim 1}
\newtheorem*{claim2}{Claim 2}
\newtheorem*{claim3}{Claim 3}

\newcommand{\R}{\mathbb{R}}
\newcommand{\W}{\mathcal{W}}

\newcommand{\C}{\mathcal{C}}

\DeclareMathOperator{\length}{length}

\DeclareMathOperator{\PH}{PH}

\DeclareMathOperator{\Stable}{S}

\usepackage{accents}

\usepackage{xcolor}
\usepackage{hyperref}
\hypersetup{
    colorlinks,
    linkcolor={red!50!black},
    citecolor={blue!50!black},
    urlcolor={blue!80!black}
}

\usepackage{quoting,xparse}

\numberwithin{equation}{section}

\begin{document}

\title[Partially hyperbolic diffeomorphisms and center fixing]{Partially hyperbolic diffeomorphisms that are center fixing}

\begin{abstract}
We show that every transitive dynamically coherent partially hyperbolic diffeomorphism with a one-dimensional center foliation $\W^c$ satisfying that $f(W)=W$ for every leaf $W\in \W^c$ is a discretized Anosov flow.
\end{abstract}

\author{Santiago Martinchich}
\address{IESTA, FCEA, Universidad de la Rep\'ublica, Uruguay
\vspace{-0.20cm}}

\address{PEDECIBA, MEC-UDELAR, Uruguay\vspace{0.10cm}}

\email{santiago.martinchich@fcea.edu.uy}
 	 	
\thanks{S. M. was partially supported by IDEX Paris-Saclay ANR-11-IDEX-0003-02, ERC project 692925 NUHGD, CSIC (UDELAR) and `Fondo Carlos Vaz Ferreira 2023´ (MEC, Uruguay).}

\maketitle


\section{Introduction}
A diffeomorphism $f:M\to M$ in a closed Riemannian manifold $M$ is called \emph{partially hyperbolic} if it preserves a continuous $Df$-invariant splitting $TM=E^s\oplus E^c \oplus E^u$ such that for some positive integer $\ell>0$ one has that 
\begin{center} $
\max\{\|Df^\ell_xv^s\|,\|Df^{-\ell}_xv^u\|\}< \frac{1}{2}$ \hspace{0.2cm} and \hspace{0.2cm}
$\|Df^{\ell}_xv^s\|<\|Df^{\ell}_xv^c\|< \|Df^{\ell}_xv^u\|$
\end{center}
for every $x\in M$ and unit vectors $v^s\in E^s(x)$, $v^c\in E^c(x)$ and $v^u\in E^u(x)$.

We denote by $\PH(M)$ the set of partially hyperbolic systems in $M$ and by $\PH_{c=1}(M)$ the ones such that $\dim(E^c)=1$.

\begin{definition}\label{defDAFnew} A map $f\in \PH_{c=1}(M)$ is called a \emph{discretized Anosov flow} if there exist a continuous flow $X_t:M\to M$, with $\frac{\partial X_t}{\partial t}|_{t=0}$ a continuous vector field without singularities, and a continuous function $\tau:M\to \R$ such that 
$$f(x)=X_{\tau(x)}(x)$$
for every $x\in M$.
\end{definition}

A map $f\in \PH(M)$ is called \emph{dynamically coherent} if there exist $f$-invariant foliations $\W^{cs}$ and $\W^{cu}$ tangent to $E^s\oplus E^c$ and $E^c\oplus E^u$, respectively. If this is the case, the foliation $\W^c$ given by the connected components of the intersection between $\W^{cs}$-leaves and $\W^{cu}$-leaves is a $f$-invariant foliation tangent to $E^c$. We say that $f\in \PH(M)$ is a \emph{dynamically coherent center fixing} map if the center foliation $\W^c$ satisfies that $f(W)=W$ for every leaf $W\in \W^c$.

It follows from \cite{Mar} that every discretized Anosov flow is a dynamically coherent center fixing map such that the leaves of the center foliation $\W^c$ are the orbits of the flow $X_t$. The goal of this article is to show the converse whenever $\W^c$ is a transitive one-dimensional foliation:

\begin{thm}\label{thm1} Suppose $f\in \PH_{c=1}(M)$ is a dynamically coherent center fixing map such that $\W^c$ has a dense leaf. Then $f$ is a discretized Anosov flow.
\end{thm}

If a dynamically coherent center fixing map $f\in \PH_{c=1}(M)$ has a point $x$ whose orbit is dense in $M$, it is immediate that $\W^c(x)$ is also dense. One obtains:

\begin{cor}\label{cor1} In $\PH_{c=1}(M)$ the class of transitive discretized Anosov flows coincides with the class of transitive  dynamically coherent center fixing maps.
\end{cor}

Theorem \ref{thm1} gives a partial answer to the following question, which remains open in general:

\begin{quest}\label{questcenterfix}
Suppose $f\in \PH_{c=1}(M)$ admits a foliation $\W^c$ tangent to $E^c$ such that $f(W)=W$ for every leaf $W\in \W^c$. Is $f$ a discretized Anosov flow?
\end{quest}

It is worth noting \cite[Question 1.3.]{Gog} where a similar question has been posed. In dimension 3, a positive answer to Question \ref{questcenterfix} was given in \cite[Theorem D]{Mar} provided $f$ is transitive. Also in \cite{BFFP}, provided the lift of $\W^c$ to the universal cover is fixed by some lift of $f$ (though with an argument specific for dimension 3).

Partially hyperbolic diffeomorphisms that are \emph{center fixing} were studied in \cite{AVW} (see also \cite[Section 8]{W}. Discretized Anosov flows have been profusely studied, at least since \cite{BW}, as the class of partially hyperbolic that naturally capture the behavior seen near the time 1 map of Anosov flow (see \cite{Mar} and references therein).

In dealing with Question \ref{questcenterfix} we found many analogies with the problem of whether every center foliation $\W^c$ by compact leaves is uniformly compact (meaning that there is uniform bound on the volume of leaves). This problem was studied in \cite{Car}, \cite{Gog} and \cite{DMM} and remains open in general.

One may ask whether every homeomorphism leaving invariant (as a set) every orbit of a continuous flow $X_t$ needs to be the time $\tau$ of $X_t$ for some continuous function $\tau:M\to \mathbb{R}$. We elaborate on this problem and give a few examples where such a function $\tau$ can not be constructed in Section \ref{sectionhomeos}, which is independent of the rest of the text.

The bulk of the article, from Section \ref{sectioncenterdisplandcenterflow} to Section \ref{sectionrhobounded}, is devoted to the proof of Theorem \ref{thm1}.

\vspace{0.3cm}
{\small \emph{Acknowledgments:} This paper was the last part of my PhD thesis. The starting point for this work was an early draft by A. Gogolev and R. Potrie. I want to thank my advisors S. Crovisier and R. Potrie. And I want to thank P. Lessa for his interest in the problem and helpful discussions. }

\section{Center flow and the center displacement function}\label{sectioncenterdisplandcenterflow}

This section initiates the proof of Theorem \ref{thm1}, which will take place until the end of Section \ref{sectionrhobounded}.

We will suppose from now on that $f\in \PH_{c=1}(M)$ admits $f$-invariant foliations $\W^{cs}$ and $\W^{cu}$ such that $f(W)=W$  for every leaf $W$ in the center foliation $\W^c:=\W^{cs}\cap \W^{cu}$. Let us suppose also that $\W^c$ has a dense leaf (though this hypothesis will only be needed in Corollary \ref{corestrellatransitive} and much of the text is independent of it, see Remark \ref{rmkhypWctransitive}). The goal is to show that $f$ must be a discretized Anosov flow.

Recall that it follows from \cite{HPS} that the bundles $E^s$ and $E^u$ uniquely integrate to $f$-invariant foliations $\W^s$ and $\W^u$, respectively. The leaves of $\W^{cs}$ and $\W^{cu}$ are $C^1$ immersed submanifolds of $M$ that are subfoliated by the leaves of $\W^s$ and $\W^u$, respectively, which are also $C^1$ immersed submanifolds. Also the leaves of $\W^c$ are $C^1$ immersed submanifolds which subfoliate the leaves of $\W^{cs}$ and $\W^{cu}$.

\begin{notation}
If $\W$ is any of these foliations, we will denote by $\W_R(x)$ the ball of radius $R>0$ and center $x$ in the leaf $\W(x)$ with respect to the intrinsic metric induced by the Riemannian metric in $M$. Moreover, if $A$ is any subset of $M$ we will denote by $\W(A)$ the saturation of $A$ by $\W$-leaves, that is, the set $\bigcup_{y\in A}\W(y)$. We will also denote by $\W_\delta(A)$ the set $\bigcup_{y\in A}\W_\delta(y)$.
\end{notation}

The first thing to point out, before properly starting the proof of Theorem \ref{thm1}, is that the center fixing problem is already solved provided there is a uniform bound on the center distance between any point and its image:

\begin{prop}[\cite{Mar}]\label{propcenterfixingL} Suppose $g\in\PH_{c=1}(M)$. The following are equivalent:
\begin{enumerate}[label=(\roman*)]
\item\label{item1prop2} The map $g$ is a discretized Anosov flow.
\item \label{item2prop2}  There exists a center foliation $\W^c$ and a constant $L>0$  such that $g(x)\in \W^c_L(x)$ for every $x\in M$.
\end{enumerate}
\end{prop}

The goal will be to show that such a constant $L>0$ exists in our setting. We will begin with a few preliminary remarks to establish the working framework.

There are two types of leaves in $\W^c$. The compact ones, that we will call \emph{circles}, and the non compact ones, that we will call \emph{lines}. As shown in the next lemma, most of the leaves of $\W^c$ are lines.

\begin{lemma}\label{lemmafinitelymanyleavesoflength<R} For every $L>0$ the number of circle leaves of $\W^c$ whose length is less than $L$ is finite. In particular, the set $\{W\in \W^c:W\text{ is compact}\}$ has at most countably many elements.
\end{lemma}
\begin{proof}
The proof of this lemma is immediate by transverse hyperbolicity in a neighborhood of each compact center leaf of $\W^c$ (every such leaf is an $f$-invariant compact submanifold that is normally hyperbolic). 

Indeed, suppose by contradiction that for some $L_0>0$ there exists an infinite number  $\{\W^c(x_n)\}_{n\geq 0}$ of distinct circle leaves of $\W^c$ whose length is less than $L_0$.

Modulo subsequence one can suppose that the sequence $(x_n)_n$ has limit $x\in M$ and that the sequence of lengths $\big(\length \W^c(x_n)\big)_n$ converges to some constant $L_1>0$. It follows that $\W^c(x)$ is compact and that $\length(\W^c(x))\leq L_1$. 

Let $\epsilon>0$ be a small constant so that $\W^s_\epsilon(y)$ and  $\W^s_\epsilon(z)$ do not intersect for every pair $y,z\in \W^c(x)$ such that $y\neq z$. And small enough so that $\W^u_\epsilon(y)$ and $\W^u_\epsilon(z)$ do not intersect for every $y,z\in \W^s_\epsilon(\W^c(x))$ such that $y\neq z$. Let $\kappa> 1$ be such that $\max\{||Df_x||,||Df^{-1}_x||\}<\kappa$ for every $x\in M$. Let $U:=\W^u_{\kappa^{-1}\epsilon}(\W^s_{\kappa^{-1}\epsilon}(\W^c(x)))$. 

Note that $f(U)\subset \W^u_\epsilon(\W^s_\epsilon(\W^c(x)))$ and that $f^{-1}(U)\subset \W^u_\epsilon(\W^s_\epsilon(\W^c(x)))$. Since $f$ contracts indefinitely stable and unstable discs for forwards and backwards iterates, respectively, it follows that $\W^c(x)=\bigcap_{k\in \mathbb{Z}}f^k(U)$. 

For every $n$ large enough the leaf $\W^c(x_n)$ is contained in $U$ since $\lim_n x_n=x$ and $\lim_n\length \W^c(x_n)=L_1$. It follows from $f(\W^c(x_n))=\W^c(x_n)$ that $\W^c(x_n)\subset \bigcap_{k\in \mathbb{Z}}f^k(U)$. That is, $\W^c(x_n)=\W^c(x)$ for every $n$ large enough. This gives us a contradiction.
\end{proof}

The next two lemmas will allow us to reduce the problem to the case where $\W^c$ is orientable and $f$ preserves the orientation of its leaves. 

\begin{lemma}\label{lemmafnDAFthenfDAF} To show that $f$ is a discretized Anosov flow is enough to show that $f^n$ is a discretized Anosov for some $n>1$.
\end{lemma}
\begin{proof} Suppose that $f^n$ is a discretized Anosov flow for some $n>1$ as in Definition \ref{defDAFnew}. It follows from \cite[Proposition 3.1]{Mar} that the orbits of the flow $X_t$ defining $f^n$ form a center foliation $\W^c$ tangent to $E^c$ and that the function $\tau$ has constant sign. 

Note first that, as $\tau$ has constant sign, if $W$ is a leaf of $\W^c$ that is not compact then $f^n$ has no fixed points in $W$. As a consequence, $f$ has no fixed points in $W$ either. In particular, $f$ preserves the orientation of $W$.

By Proposition \ref{propcenterfixingL} there exists $L>0$ such that $f^n(x)\in \W^c_L(x)$ for every $x\in M$. In case $\W^c(x)$ is not compact, the fact that $f$ preserves the orientation of $\W^c(x)$ implies that $f(x)$ lies in the center interval $[x,f^n(x)]_c$ joining $x$ with $f^n(x)$. In particular, $f(x)$ lies in $\W^c_L(x)$.

By Lemma \ref{lemmafinitelymanyleavesoflength<R} there exists at most countably many compact leaves of $\W^c$. So given $x$ in a compact leaf of $\W^c$ one can consider a sequence $x_n$ converging to $x$ so that $\W^c(x_n)$ is not compact for every $n$. As $f(x_n)$ belongs to $\W^c_L(x_n)$ for every $n$ and the sequence $f(x_n)$ tends to $f(x)$ one obtains that $f(x)$ lies in $\W^c_L(x)$. 

We have shown that $f(x)\in \W^c_L(x)$ for every $x\in M$. By Proposition \ref{propcenterfixingL} we conclude that $f$ is a discretized Anosov flow.
\end{proof}

As mentioned before, the above lemma has the following immediate consequence:

\begin{lemma}\label{lemmaWcorientableandfpreservesorient} To show that $f$ is a discretized Anosov flow we can assume that the foliation $\W^c$ is orientable and that $f$ preserves the orientation of $\W^c$-leaves.
\end{lemma}
\begin{proof} Consider a double cover $\bar{M}$ of $M$ so that $\W^c$ lifts to an orientable foliation $\bar{\W}^c$. Let $\bar{f}:\bar{M}\to \bar{M}$ be a lift of $f$. A priori $\bar{f}$ may not individually fix every leaf of $\tilde{\W}^c$. However, $\bar{f}^2$  does. Moreover, $\bar{f}^4$ preserves the orientations of these leaves. 

Let $g:=\bar{f}^4$ and suppose that $g$ is a discretized Anosov flow. It follows from Proposition \ref{propcenterfixingL} that there exists $L>0$ so that $g(x)$ lies in $\bar{\W}^c_L(\tilde{x})$ for every $\bar{x}\in \bar{M}$. Then $f^4(x)$ lies in $\W^c_L(x)$ for every $x\in M$. It follows that $f^4$ is a discretized Anosov flow. By Lemma \ref{lemmafnDAFthenfDAF}, $f$ itself is a discretized Anosov flow. This concludes the proof.
\end{proof}

In line with the previous lemma, we will assume from now on that the foliation $\W^c$ is orientable and that $f$ preserves the orientation of $\W^c$ leaves. 

In particular, this allows us to consider a flow along $\W^c$ leaves:

\begin{definition}[Center flow] Let $X^c_t:M\to M$ be a flow by arc-length whose orbits  are the leaves of $\W^c$.
\end{definition}

We will also consider the following function:

\begin{definition}[Center displacement function]
Let us define $\rho:M\to \mathbb{R}$ to be $$\rho(x):=d_c(x,f(x))$$ for every $x\in M$.
\end{definition}

\begin{remark}\label{rmkDAFiftaubounded}
It is immediate from Proposition \ref{propcenterfixingL} that if $\rho$ is bounded then $f$ is a discretized Anosov flow. This will be shown in Corollary \ref{corfinal} and will be the goal for the rest of the text.
\end{remark}

Recall that a real valued function $F$ in $M$ is called \emph{lower semicontinuous} if for every sequence $(x_n)_n$ converging to $x$ one has that $\liminf_n F(x_n)\geq F(x)$.

\begin{prop}\label{proprhosemicont}
The function $\rho$ is lower semicontinuous.
\end{prop}
\begin{proof}
Suppose $(x_n)_n$ is a sequence converging to $x$. In case $\liminf_n \rho(x_n)=+\infty$ then there is nothing to proof. Otherwise, one can consider a subsequence so that the limit inferior is in fact a limit. By a slight abuse of notation let us denote this subsequence also $(x_n)_n$.

For every $n$ let $\gamma_n:[0,1]\to \W^c(x_n)$ denote a $C^1$ curve parametrized by arc-length so that $\length(\gamma_n)=\rho(x_n)$. The sequence $\dot{\gamma_n}(0)$ accumulates in a unitary vector $v_c$ in $E^c(x)$. Up to taking a subsequence (in case $-v_c$ is also an accumulation point), let us suppose that $\dot{\gamma_n}(0)$ converges to $v_c$. 

Let $\gamma:[0,1]\to \W^c(x)$ denote the $C^1$ curve parametrized by arc-length so that $\dot{\gamma}(0)=v_c$. It follows that $\gamma_n$ converges with $n$ to $\gamma$ in the $C^1$ topology. In particular, $\length(\gamma_n)$ converges to $\length(\gamma)$.

Since $\gamma_n(1)$ converges with $n$ to $\gamma(1)$ and $\gamma_n(1)=f(x_n)$ for every $n$ it follows that $\gamma(1)=f(x)$. As $\gamma$ is a $C^1$ curve in $\W^c(x)$ joining $x$ to $f(x)$ it follows that $\rho(x)\leq \length(\gamma)$. Since $\length(\gamma_n)=\rho(x_n)$ converges with $n$ to $\length(\gamma)$ we obtain that $\rho(x)\leq \lim_n \rho(x_n)$ as desired. 
\end{proof}

\begin{remark}
It is worth noting that we can not expect to show that, in general, $\rho$ has to be continuous at every point of $M$. For example, if $f$ is the time 1 map of an Anosov flow $Y_t:M\to M$, and $Y_t$ is parametrized by arc-length, then $\rho$ will only be continuous in the complement of the set of periodic orbits of period smaller than $2$.
\end{remark}

\begin{definition} Let us denote by $X\subset M$ the set of continuity points of $\rho$ in $M$. Namely $$X=\{x\in M \mid \rho \text{ is continuous at }x\}.$$

Let us denote by $Y\subset X$ the set of continuity points $x$ such that $\W^c(x)$ is a line. Namely $$Y=\{x\in X \mid \W^c(x) \text{ is a line}\}.$$
\end{definition}

\begin{remark} By the semicontinuity of $\rho$ it follows by classical arguments that $X$ is a \emph{residual subset of $M$} (meaning that it is a countable intersection of open and dense subsets of $M$).

Moreover, it is not difficult to show that in our context $X$ must in fact contain an open a dense subset of $M$. We do not prove this statement since we will not need it later.
\end{remark}

Recall that a topological space is \emph{locally path connected} if every point has a local basis made of path connected open sets.

\begin{prop}\label{propYresidualWcsatpathconn}
The set $Y$ is residual in $M$, saturated by leaves of $\W^c$ and locally path connected. 
\end{prop}

\begin{proof}
By Lemma \ref{lemmafinitelymanyleavesoflength<R} the set $\{W\in \W^c:W\text{ is compact}\}$ has countably many elements. Moreover, each of these elements is a nowhere dense subset of $M$. Since $X$ is residual in $M$ it follows that $Y=X\setminus \big(\bigcup_{W\in \W^c, W\text{ is compact}}W\big)$ is also residual in $M$.

Let us see now that $Y$ is saturated by leaves of $\W^c$. Fix $\delta>0$. 
The goal will be to see that $\W^c_\delta(x)$ is a subset of $Y$ for every $x\in Y$.

Let $x\in Y$. The leaf $\W^c(x)$ is a line. Let $[x,f(x)]_c$ denote the segment joining $x$ with $f(x)$ inside $\W^c(x)$. Let $U$ be an open $\W^c$-tubular neighborhood of $[x,f(x)]_c$. More precisely, $U$ is the image of  a certain homeomorphism over its image $\Psi:\R^{\dim(M)-1}\times \R\to M$ such that $\Psi(p\times \R)$ is contained in a leaf of $\W^c$ for every $p\in \R^{\dim(M)-1}$ and $[x,f(x)]_c\subset \Psi(0\times \R)$.

The open tubular neighborhood $U$ can be considered  `long and thin' enough, and a ball $B_{\epsilon_x}(x)\subset U$ for some $\epsilon_x>0$ considered small enough, so that for every $y\in B_{\epsilon_x}(x)$ one has that
\begin{enumerate}[label=(\roman*)]
\item\label{item1YWcsat} $\W^c_{\rho(x)+2\delta}(y)$ is contained in a center plaque of $U$,
\item\label{item2YWcsat} $f(\W^c_{2\delta}(y))$ is contained in a center plaque of $U$ and

\item\label{item3YWcsat} $|\rho(y)-\rho(x)|<\delta$.
\end{enumerate}

We claim first that every point in $B_{\epsilon_x}(x)$ is a continuity point of $\rho$. Indeed, for every $y\in B_{\epsilon_x}(x)$ it follows from \ref{item3YWcsat} that the set $\W^c_{\rho(y)+\delta}(y)$ is a subset of $\W^c_{\rho(x)+2\delta}(y)$. By \ref{item1YWcsat} the set $\W^c_{\rho(x)+2\delta}(y)$ is contained in a center plaque of $U$, and $f(y)$ lies in $\W^c_{\rho(y)+\delta}(y)$ by definition of $\rho$, so $y$ and $f(y)$ must lie in the same center plaque of $U$. 

For every $y\in B_{\epsilon_x}(x)$ let $[y,f(y)]_c$ be the center segment in $U$ joining $y$ with $f(y)$. Since $\W^c_{\rho(y)+\delta}(y) \subset U$ it follows that $\rho(y)=\length [y,f(y)]_c$ for every $y\in B_{\epsilon_x}(x)$.

Given a sequence $(y_n)$ in $B_{\epsilon_x}(x)$ converging to $y\in B_{\epsilon_x}(x)$, the center plaque of $U$ containing $y_n$ approaches the one containing $y$. Since $f(y_n)$ tends to $f(y)$ it follows that $[y_n,f(y_n)]_c$ converges to $[y,f(y)]_c$ in the Hausdorff topology. Then $\rho(y_n)$ converges to $\rho(y)$. This proves the claim. 

Suppose now that $z$ is a point in $\W^c_\delta(x)$ and $(z_n)$ is a sequence converging to $z$. Let $(x_n)$ be a sequence converging to $x$ so that $z_n$ belongs to $\W^c_\delta(x_n)$ for every $n$. For every $n$ large enough the point $x_n$ lies in $B_{\epsilon_x}(x)$, so $x_n$ and $f(x_n)$ lie in the same center plaque of $U$ as it was seen above.

Modulo dropping the first iterates of the sequence suppose without loss of generality that $x_n$ lies in $B_{\epsilon_x}(x)$ for every $n$. Then by \ref{item1YWcsat} and \ref{item2YWcsat} one has that $\W^c_\delta(x_n)\subset U$ and $f(\W^c_\delta(x_n))\subset U$ for every $n$.

Since $[x_n,f(x_n)]_c$ is contained in a center plaque of $U$ it follows that $\W^c_\delta(x_n)\cup [x_n,f(x_n)]_c \cup f(\W^c_\delta(x_n))$ is a center segment that is also contained in the same center plaque of $U$. Since $z_n$ lies $\W^c_\delta(x_n)$ and $f(z_n)$ lies in $f(\W^c_\delta(x_n))$ one obtains that  $z_n$ and $f(z_n)$ also lie in the same center plaque of $U$ for every $n$.

Let $[z_n,f(z_n)]_c$ denote the center segment in $U$ joining $z_n$ with $f(z_n)$ for every $n$. Since $x_n$ lies in $B_{\epsilon_x}(x)$ we have shown above that $\W^c_{\rho(x_n)+\delta}(x_n)$ is contained in a center plaque of $U$. So $z_n\in \W^c_\delta(x_n)$ implies that $\W^c_{\rho(z_n)}(z_n)$ is contained in the same center plaque of $U$. This shows that $\rho(z_n)$ is equal to the length of $[z_n,f(z_n)]_c$ for every $n$.

As above, the sequence $[z_n,f(z_n)]_c$ needs to converge to $[z,f(z)]_c$ in the Hausdorff topology. One  obtains that $\rho(z_n)$ converges to $\rho(z)$. That is, $z$ is a continuity point of $\rho$.

We have shown that for every $x\in Y$ the set $\W^c_\delta(x)$ is a subset of $Y$ for some uniform $\delta>0$. We conclude that $Y$ is saturated by leaves of $\W^c$.

Locally, the set $Y$ is an open subset of $M$ minus, at most, countably many center plaques corresponding to circle leaves of $\W^c$. It is immediate from this that $Y$ is locally path connected.
\end{proof}

\section{No $s$ or $u$ self-accumulation of center leaves}

The goal of this section is to show in Proposition \ref{propnouautobads2} that there is no self-accumulation of center leaves inside leaves of $\W^{cs}$ and $\W^{cu}$.

Let us start by pointing out two lemmas that will be needed.

\begin{lemma}\label{lemmafinitelineswithfixedpointsandbounded} There are at most finitely many line leaves of $\W^c$ having fixed points. Moreover, in every such a leaf the fixed points lie in a bounded interval of the leaf.
\end{lemma}
\begin{proof}
Given $x\in M$ let $U_x$ be a small $\W^c$ foliation box neighborhood of $x$. By transverse hyperbolicity (see for example Lemma \ref{lemmafinitelymanyleavesoflength<R} for more details) the set $U_x$ contains at most one center plaque $I_x$ such that $f(I_x)\cap I_x\neq \emptyset$. Consider a finite subcover $\{U_{x_1},\ldots,U_{x_N}\}$ of $M$. Every fixed point of $f$ must lie in $I_{x_1}\cup \ldots \cup I_{x_N}$. This proves the lemma.
\end{proof}

\begin{lemma}\label{lemmatwocircleleaves}
Two distinct circle leaves of $\W^c$ can not intersect the same leaf of $\W^s$. Moreover, every circle leaf of $\W^c$ intersects every leaf of $\W^s$ in at most one point. Analogously for $\W^u$ leaves.
\end{lemma}
\begin{proof}
Suppose by contradiction that two distinct circle leaves $W,W'\in \W^c$ intersect the same leaf of $\W^s$ in points $x\in W$ and $y\in W'$. By iterating forwards one obtains that $d(f^n(x),f^n(y))$ tends to zero. By the center fixing property $f^n(x)$ and $f^n(y)$ are points in $W$ and $W'$, respectively, for every $n$. This contradicts the fact that the disjoint compact sets $W$ and $W'$ are at positive distance from each other. 

Given a circle leaf $W\in \W^c$ there exists $\delta>0$ such that $\bigcup_{x\in W}\W^s_\delta(x)\setminus \{x\}$ is disjoint from $W$, otherwise $W$ would not be compact. If one supposes, by contradiction, that $W$ intersects the same leaf of $\W^s$ in two different points $x$ and $y$ then $d(f^n(x),f^n(y))$ will be smaller than $\delta$ for some large enough $n>0$. Since $f^n(x)$ and $f^n(y)$ are points in $W$ this contradicts that $\W^s_\delta(f^n(x))\setminus \{f^n(x)\}$ is disjoint from $W$.

Analogously for $\W^u$ leaves by iterating backwards by $f$.
\end{proof}

\begin{notation} From now on, given $x\in M$ and $y\in \W^c(x)$ such that $\W^c(x)$ is a line, let $(x,y)_c$ denote the open center segment from $x$ to $y$ inside $\W^c(x)$. In case $\W^c(x)$ is a circle let $(x,y)_c$ denote the center segment joining $x$ and $y$, and containing $X^c_t(x)$ for every small enough positive $t$. 

Moreover, in case $\W^c(x)$ is a line let $(x,-\infty)_c$ and $(x,+\infty)_c$ denote the connected component of $\W^c(x)\setminus x$ containing negative and positive $X^c_t$-iterates of $x$, respectively. 

Analogously let us define the closed $[x,y]_c$ and half-open $[x,y)_c$ and $(x,y]_c$ center segments.
\end{notation}

The following proposition state the no `self-accumulation' of $\W^c$ leaves inside leaves of $\W^{cs}$ and $\W^{cu}$.

\begin{prop}\label{propnouautobads2}
For every $x\in M$ there exists $\epsilon>0$ so that $\W^c(x)\cap\W^s_\epsilon(x)=\{x\}$ and $\W^c(x)\cap\W^u_\epsilon(x)=\{x\}$.
\end{prop}

\begin{proof} Given $x\in M$, let us see that there exists $\epsilon>0$ so that $\W^c(x)=\{x\}\cap \W^u_\epsilon(x)$. For $s$ leaves the reasoning is analogous.

Suppose, by contradiction, that there exists a sequence $(y_n)_{n\geq 0}$ in $\W^c(x) \cap (\W^u(x)\setminus \{x\})$ converging to $x$ in $\W^u(x)$.

Note first that $\W^c(x)$ needs to be a line leaf of $\W^c$. Indeed, $\W^c(x)$ must contain infinitely many disjoint center arcs whose lengths are bounded from below, one for each distinct point in $(y_n)_{n\geq 0}$ close enough to $x$ in $\W^u(x)$. So the length of $\W^c(x)$ has to be unbounded.

Moreover, it follows by Lemma \ref{lemmatwocircleleaves} that at most one circle leaf of $\W^c$ contain points from $(y_n)_n$ and, by a similar argument as before, these points can not accumulate in $x$. One obtains that there exists $\delta>0$ such that $\W^u_\delta(x)$ contains no circle leaf of $\W^c(x)$. Let us suppose, without loss of generality, that $y_n$ lies in $\W^u_\delta(x)$ for every $n\geq 0$.

As $\W^c(x)$ is a line it follows that $y_n\xrightarrow{n} \infty$ in $\W^c(x)$, meaning that for every $R>0$ the point $y_n$ does not lie in $\W^c_R(x)$ for every $n$ large enough. Modulo subsequence, suppose without loss of generality that $y_n$ tends to $+\infty$ in $\W^c(x)$, meaning that for every $R>0$ the point $y_n$ lies in $[x,+\infty)_c\setminus \W^c_R(x)$ for every $n$ large enough. It will be clear from the proof that if $y_n$ tends to $-\infty$ the arguments are analogous to get to a contradiction.

By Lemma \ref{lemmafinitelineswithfixedpointsandbounded} there exists $p\in \W^c(x)$ the `last' fixed point of $f$ in $\W^c(x)$ (if any fixed point in $\W^c(x)$ exists), meaning that $(p,+\infty)_c$ contains no fixed points. In case $\W^c(x)$ has no fixed point of $f$ let $p=-\infty$. Note that for every $x'$ in $\W^c(x)$ one can project every large enough element of the sequence $(y_n)_{n\geq 0}$ by center holonomy to obtain a sequence in $\W^c(x) \cap (\W^u_\delta(x')\setminus \{x'\})$ converging to $x'$. We can assume then, without loss of generality, that $x$ lies in $(p,+\infty)_c$.

Note that the half-open center segment $[x,f(x))_c$ is a fundamental domain for the dynamics $f|_{(p,+\infty)_c}:(p,+\infty)_c\to (p,+\infty)_c$. So for every $n\geq 0$ there exists $x_n\in [x,f(x))_c$ and $k_n\in \mathbb{Z}$ such that $y_n=f^{k_n}(x_n)$. Note that either $k_n\to -\infty$ or $k_n\to +\infty$ as $n$ tends to $+\infty$. 

Again, without loss of generality, let us consider $\delta>0$ small enough so that so that $\W^u_\delta(w)\cap \W^u_\delta(w')=\emptyset$ for every $w,w'\in [x,f(x)]_c$ such that $w\neq w'$. And small enough so that $\W^s_\delta([x,f(x))_c)$, its image by $f$ and its image by $f^{-1}$ are two-by-two disjoint.

For $\delta'>0$ small enough let $H^c:\W^u_{\delta'}(x)\to \W^u_\delta(f(x))$ be a \emph{center holonomy map} so that $I_z:=[z,H^c(z)]_c$ is a center segment in $\W^u_{\delta}([x,f(x)]_c)$ for every $z\in \W^u_{\delta'}(x)$. Let $U$ be the set $\bigcup_{z\in \W^u_{\delta'}(x)}I_z$. Note that $U$ is a foliation box neighborhood for the foliations $\W^c$ and $\W^u$ restricted to $\W^{cu}(x)$.

In $U$ consider $\pi^u:U\to \W^u_{\delta'}(x)$ the projection along centers so that $\pi^u(I_z)=z$ for every $z\in \W^u_{\delta'}(x)$. And for every $n$ let $H^c_n:\W^u_{\delta'}(x) \to \W^u_{\delta}(x_n)$ be such that $H^c_n(z)$ is the intersection of $I_z$ with $\W^u_{\delta}(x_n)$ for every $z\in \W^u_{\delta'}(x)$.

Let $\delta''>0$ be such that $\W^u_{\delta''}(z)$ is contained in $U$ for every $z$ in the segment $[x,f(x)]_c$. In particular, for every $n$ the inverse map $(H^c_n)^{-1}$ is well defined from $\W^u_{\delta''}(x_n)$ to $\W^u_{\delta'}(x)$.

Consider $N$ large enough so that $y_N$ lies in $\W^u_{\delta'}(x)$. And so that $f^{-|k_N|}$ contracts distances enough so that $\W^u_{\delta'}(w)$ is sent inside $\W^u_{\delta''/2}(f^{-|k_N|}(w))$ for every $w\in M$.

In case $k_N>0$ the map $(H^c_N)^{-1} \circ f^{-k_N}$ is a continuous map that sends $\W^u_{\delta'}(x)$ strictly inside itself. In case $k_N<0$ then $f^{k_N} \circ H^c_N$ is a continuous map that sends $\W^u_{\delta'}(x)$ strictly inside itself. Covering both scenarios at the same time, let $z_0$ denote the fixed point of this map. In the first case $f^{-k_N}(z_0)$ lies in $I_{z_0}$ and in the second case $f^{k_N}(z_0)$ lies in $I_{z_0}$. See Figure \ref{fig1} for a schematic illustration.

\begin{figure}[htb]
\def\svgwidth{0.8\textwidth}
\begin{center} 
{\scriptsize
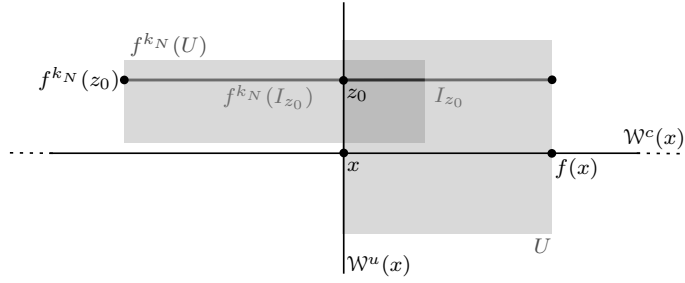
}
\caption{A large iterate of $f$ sends the $\W^c$-foliation box $U$ transverse to itself, so the iterate of some plaque $I_{z_0}$ of $U$ has to intersect itself. As a consequence, $\W^c(x_0)$ needs to be compact by the center fixing property.}\label{fig1} 
\end{center}
\end{figure}

Note that $f|_{\W^c(z_0)}:\W^c(z_0)\to \W^c(z_0)$ is an orientation preserving homeomorphism satisfying that $I_{z_0}$ is contained in $[x_0,f(x_0)]_c$ or $[x_0,f^{-1}(x_0)]_c$ because $\W^s_\delta([x,f(x))_c)$, its image by $f$ and its image by $f^{-1}$ were taken two-by-two disjoint. Since $f^{-k_N}(z_0)$ or $f^{k_N}(z_0)$ lies in $I_{z_0}$ for some $|k_N|$ large the only possibility is that $\W^c(z_0)$ is a circle.

We obtain a circle leaf of $\W^c$ intersecting $\W^s_\delta(x)$ and this gives us a contradiction.
\end{proof}

\section{Continuity of $\rho$ in lines of $\W^c$}

The goal of this section is to show that $\rho$ is continuous at every point in a line leaf of $\W^c$. This is attained in Corollary \ref{corestrellatransitive}.

\begin{definition}
We say that $x$ is a \emph{$s$-continuity point of $\rho$} if there exists $\delta>0$ such that $\rho$ restricted to $\W^s_\delta(x)$ is continuous. Analogously we define a \emph{$u$-continuity point of $\rho$}.
\end{definition}

Because of the following lemma, which is a direct consequence of dynamical coherence, it is enough to study the continuity of $\rho$ in restriction to $\W^s$ and $\W^u$ leaves.

\begin{lemma}\label{lemmasanducontpointtheninX}
Suppose $x\in M$ such that $\W^c(x)$ is a line leaf. If $x$ is a $s$ and $u$-continuity point of $\rho$, then $\rho$ is continuous at $x$.
\end{lemma}
\begin{proof}

Let $\gamma:[0,1]\to \W^c(x)$ be a homeomorphism from $[0,1]$ to the center segment $[x,f(x)]_c$ so that $\gamma(0)=x$ and $\gamma(1)=f(x)$. Since $x$ is a $s$ and $u$-continuity point of $\rho$ there exists $\delta>0$ such that for every $y\in \W^s_\delta(x)$ there exists $\gamma_y^s:[0,1]\to \W^c(y)$ satisfying $\gamma^s_y(0)=y$, $\gamma^s_y(t)\in\W^s(\gamma(t))$ for every $t\in [0,1]$ and $\gamma^s_y(1)=f(y)$. And such that for every $z\in \W^u_\delta(x)$ there exists $\gamma_z^u:[0,1]\to \W^c(y)$ satisfying $\gamma^u_z(0)=z$, $\gamma^u_z(t)\in\W^u(\gamma(t))$ for every $t\in [0,1]$ and $\gamma^u_z(1)=f(z)$.

Let $\epsilon_0>0$ be a small constant so that at scale $\epsilon_0$ one has local product structure for the invariant foliations of $f$. More precisely, let $\epsilon_0>0$ be small enough so that there exists a constant $C>0$ such that for every $x,y\in M$ satisfying $\delta':=d(x,y)<\epsilon_0$ one has that $\W_{C\delta'}^s(x)$ intersects $\W_{C\delta'}^{cu}(y)$, and $\W_{C\delta'}^u(x)$ intersects $\W_{C\delta'}^{cs}(y)$, and each of these intersection points is unique. The constant $C>0$ is needed to take into account the angles between the bundles.

One can consider $\delta>0$ small enough so that for every $y\in \W^s_\delta(x)$ and $z\in \W^u_\delta(x)$ one has $\gamma^s_y(t)\in\W^s_{\epsilon_0/2}(\gamma(t))$ and $\gamma^u_z(t)\in\W^u_{\epsilon_0/2}(\gamma(t))$ for every $t\in [0,1]$.

Suppose $(x_n)$ is a sequence converging to $x$ and let $\delta_n$ denote the distance $d(x,x_n)$ for every $n$. Suppose, modulo subsequence, that $\delta_n$ is smaller than $C^{-1}\delta$ for every $n$. Let $x_n^s$ be the point of intersection of $\W^s_{C\delta_n}(x)$ and $\W^{cu}_{C\delta_n}(x_n)$. And $x_n^u$ the point of intersection of $\W^u_{C\delta_n}(x)$ and $\W^{cs}_{C\delta_n}(x_n)$. 

For simplicity in the notation, for every $n$ let $\gamma_n^s$ and $\gamma_n^u$ denote the curves $\gamma^s_{x_n^s}$ and $\gamma^u_{x_n^u}$, respectively. We have that $\epsilon_n(t):=d(\gamma_n^s(t),\gamma_n^u(t))<\epsilon_0$ for every $t\in [0,1]$. It follows that $\W^{cu}_{C\epsilon_n(t)}(\gamma_n^s(t))$ and $\W^{cs}_{C\epsilon_n(t)}(\gamma_n^u(t))$ intersect for every $t\in [0,1]$.  Moreover, by dynamical coherence, one can construct a homeomorphism over its image $\gamma_n:[0,1]\to \W^c(x_n)$ such that $\gamma_n(0)=x_n$ and $\gamma_n(t)$ lies in  the intersection of $\W^{cu}_{C\epsilon_n(t)}(\gamma_n^s(t))$ and $\W^{cs}_{C\epsilon_n(t)}(\gamma_n^u(t))$ for every $t\in [0,1]$. 

Again, by dynamical coherence, it follows from $f(x_n^s)=\gamma_n^s(1)$ and $f(x_n^u)=\gamma_n^u(1)$ that $f(x_n)$ lies in the intersection of $\W^{cu}_{C\epsilon_n(1)}(\gamma_n^s(1))$ and $\W^{cs}_{C\epsilon_n(1)}(\gamma_n^u(1))$. One can chose $\gamma_n$ so that, in addition to the properties from the previous paragraph, it satisfies that $\gamma_n(1)=f(x_n)$.

Since the sequences $(x_n^s)$ and $(x_n^u)$ converge to $x$ it is immediate that $\gamma_n^s$ and $\gamma_n^u$ converge in the $C^0$ topology to $\gamma$. As a consequence, since $\delta_n$ tends to $0$, then $\gamma_n$ converges in the $C^0$ topology to $\gamma$ as well.

The maps $\gamma_n:[0,1]\to \W^c(x_n)$ form a sequence of homeomorphisms over its image joining $x_n$ with $f(x_n)$ and converging $C^0$ to $\gamma$. One obtains that $\lim_n \length(\gamma_n)=\length(\gamma)$. Since $\rho(x)=\length(\gamma)$, then $\liminf_n \rho(x_n)\leq \rho(x)$. By Proposition \ref{proprhosemicont} one has the converse inequality. It follows that $\lim_n\rho(x_n)=\rho(x)$. This shows that $x$ is a continuity point of $\rho$.
\end{proof}

The next proposition, which is a consequence of the no self-accumulation of centers leaves inside $cs$ or $cu$ leaves, is a key fact that  indicates that potential discontinuities of $\rho$ in the $s$ or $u$ direction at line leaves of $\W^c$ need some separation of center leaves in order to occur.

\begin{definition} For every $x\in M$ and $\epsilon>0$  we define  $$\Stable^+_\epsilon(x)=\{y\mid \exists \text{ homeo } h:\R_{\geq 0}\to \R_{\geq 0} \text{ s.t. }d(X^c_t(x),X^c_{h(t)}(y))\leq \epsilon \text{ }\forall \text{ }t\geq 0\}$$ the \emph{forwards $\epsilon$-stable set of $x$} for $X^c_t$. Analogously we define $\Stable^-_\epsilon(x)$ the \emph{backwards $\epsilon$-stable set of $x$} for $X^c_t$ by looking at $t\leq 0$.
\end{definition}

\begin{prop}\label{propcorofnoautobads}
Suppose $x\in M$ in a line leaf of $\W^c$ satisfies that for every $\epsilon>0$ there exists $\delta>0$ such that $\W^s_\delta(x)\subset \Stable^+_\epsilon(x)$ or $\W^s_\delta(x)\subset \Stable^-_\epsilon(x)$. Then $x$ is a $s$-continuity point for $\rho$. Analogously for $u$-continuity points.
\end{prop}
\begin{proof} The rough idea is the following: If we assume that $x$ is not a $s$-continuity point for $\rho$, then for every $\delta>0$ there would exists $z_0$ in $\W^s_\delta(x)$ such that the center leaf through $z_0$ intersects $\W^s_\delta(x)$ twice (in order to get the non-continuity of $\rho$, see Figure \ref{fig2}). If we assume that the center leaf through $x$ follows closely the one through $z_0$ this will give a self-accumulation of $\W^c(x)$ into itself, thus a contradiction with Proposition \ref{propnouautobads2}. 

Let us write this more precisely. Let us consider $r>0$ small enough so that for every pair $y\neq y'$ in $[x,f(x)]_c$ one has that $\W^s_{r}(y)\cap \W^s_{r}(y')=\emptyset$. Note that a priori we can not rule out the case $f(x)=x$. If this is the case, $[x,f(x)]_c$ will simply denote the singleton $\{x\}$. 

Let $r'>0$ be small enough so that one can define $H^c:\W^s_{r'}(x)\to \W^s_{r}(f(x))$ the center holonomy map so that $[z,H^c(z)]_c$ is a center segment in $\W^s_{r}([x,f(x)]_c)$ for every $z\in \W^s_{r'}(x)$. In case $f(x)=x$ then $H$ will denote the identity map and $[z,H(z)]_c$ will be the singleton $\{z\}$ for every $z\in \W^s_{r'}(x)$. One has that $U:= \bigcup_{z\in \W^s_{r'}(x)}(z,H^c(z))_c$ is a foliation box neighborhood for $\W^c$ containing $(x,f(x))_c$ as plaque if $f(x)\neq x$. Otherwise, if $f(x)=x$, let $U$ denote the set $\W^s_{r'}(x)$.

By Proposition \ref{propnouautobads2} there exists $\epsilon>0$ such that $\W^c(x)\cap \W^s_\epsilon(x)=\{x\}$. By hypothesis, there exists $\delta\in (0,r']$ such that $\W^s_\delta(x)\subset \Stable^+_{\epsilon/2}(x)$ or $\W^s_\delta(x)\subset \Stable^-_{\epsilon/2}(x)$. Let us consider $\delta'\in (0,\delta]$ so that $f(\W^s_{\delta'}(x))$ is contained in $H^c(\W^s_\delta(x))$.

In case $H^c(z)=f(z)$ for every $z\in \W^s_{\delta'}(x)$ one obtains that $\rho$ is continuous in restriction to $\W^s_{\delta'}(x)$, showing that $x$ is a $s$-continuity point.

Let us suppose, by contraction, that there exists $z_0$ in $\W^s_{\delta'}(x)$ such that $H^c(z_0)\neq f(z_0)$. Let $z_1$ denote the point in $\W^s_\delta(x)$ such that $H^c(z_1)=f(z_0)$. One has that $z_1$ is a point in $\W^c(z_0)$ different from $z_0$. See Figure \ref{fig2}. 

\begin{figure}[htb]
\def\svgwidth{0.8\textwidth}
\begin{center} 
{\scriptsize
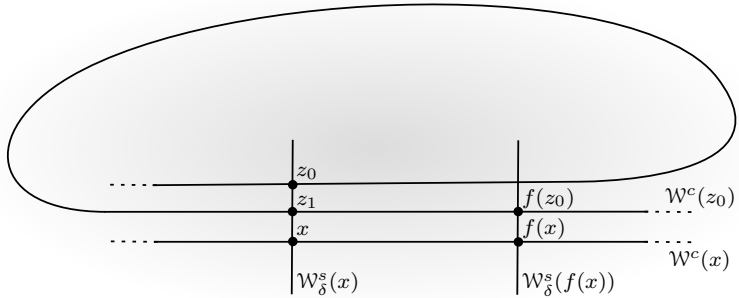
}
\caption{As close as wanted to a point $x$ where the function $z
\mapsto d_c(z,f(z))$ is not continuous one can find a point $z_0$ such that $f(z_0)$ is not given by a center holonomy map along the center curve from $x$ to $f(x)$}\label{fig2} 
\end{center}
\end{figure}
	
Let us suppose, without loss of generality, that $\W^s_\delta(x)\subset \Stable^+_{\epsilon/2}(x)$. It will be clear from the proof that the case $\W^s_\delta(x)\subset \Stable^-_{\epsilon/2}(x)$ is analogous. 
 
In case $z_1$ is equal to $X^c_T(z_0)$ for some $T>0$ there exists $h:[0,T]\to \R$ such that $d_s(X^c(z_0),X^c_{h(t)}(x))\leq \epsilon/2$ for every $t\in [0,T]$. Otherwise, if $z_0$ is equal to $X^c_T(z_1)$ for some $T>0$, then there exists $h:[0,T]\to \R$ such that $d_s(X^c(z_1),X^c_{h(t)}(x))\leq \epsilon/2$ for every $t\in [0,T]$. In either case, $X^c_{h(T)}(x)$ is a point in $\W^s_\epsilon(x)$ which is different from $x$, otherwise $\W^c(x)$ would be a circle leaf. One obtains that $\W^c(x)$ intersects $\W^s_\epsilon(x)\setminus \{x\}$ and this gives us a contradiction.

Analogously for $u$-continuity points.
\end{proof}

Another consequence of the no self-accumulation of centers leaves inside $cs$ or $cu$ leaves is that $X^c_t$ needs to be an expansive flow.

\begin{definition}[See for example \cite{BW72}]
Given a non-singular flow $X_t:X\to X$ in a metric space $X$ and a constant $\epsilon>0$ the flow $X_t$ is said to be \emph{$\epsilon$-expansive} if for every $x,y\in X$ and $h:\R\to \R$ an increasing homeomorphism with $h(0)=0$ satisfying $$d(X_t(x),X_{h(t)}(y))<\epsilon \,\,\,\forall t\in\R$$ one has that $x$ and $y$ lie in a same segment $\eta$ of $X_t$-orbit such that $\eta\subset B_{\epsilon}(x)$.
\end{definition}

\begin{prop}\label{propexpansiveflow} There exists $\epsilon_0>0$ so that the flow $X^c_t$ is $\epsilon_0$-expansive.
\end{prop}
\begin{proof}
Let $\epsilon_0>0$ be a small constant so that at scale $\epsilon_0$ one has local product structure for the invariant foliations of $f$. More precisely, let $\epsilon_0>0$ be small enough so that there exists $C>0$ such that for every $x,y\in M$ satisfying $d(x,y)<\epsilon_0$ one has that $\W_{C\epsilon_0}^s(x)$ intersects $\W_{C\epsilon_0}^{cu}(y)$, and $\W_{C\epsilon_0}^u(x)$ intersects $\W_{C\epsilon_0}^{cs}(y)$, and each of these intersection points is unique. And, moreover, small enough so that if $d_{cs}(x,y)<C\epsilon_0$ or $d_{cu}(x,y)<C\epsilon_0$, then $\W_{C^2\epsilon_0}^c(x)$ intersects $\W_{C^2\epsilon_0}^s(y)$ or $\W_{C^2\epsilon_0}^c(x)$ intersects $\W_{C^2\epsilon_0}^u(y)$, respectively, and this intersection points are also unique. 

We claim first that, by dynamical coherence, it is enough to show that $X^c_t$ is $C\epsilon_0$-expansive in restriction to leaves of $\W^{cu}$ and $\W^{cs}$ for their intrinsic topology. 

Let us show the claim. Assume that $X^c_t$ is $C\epsilon_0$-expansive in restriction to leaves of $\W^{cu}$ and $\W^{cs}$ and suppose that $x$ and $y$ are two points in $M$ such that there exists $h:\R\to \R$ an increasing homeomorphism satisfying $h(0)=0$ and $d(X^c_t(x),X^c_{h(t)}(y))<\epsilon_0$  for every $t\in \R$. It follows that $\W^u_{C\epsilon_0}(X^c_{h(t)}(y))$ must intersect $\W^{cs}_{C\epsilon_0}(X^c_t(x))$ for every $t\in\R$, and this intersection point is unique. 

Let $y_{cs}:=\W^u_{C\epsilon_0}(y)\cap \W^{cs}_{C\epsilon_0}(x)$. It follows that there exists an increasing homeomorphisms $h_{cs}:\R\to \R$  satisfying that $h_{cs}(0)=0$ and $d_{cs}(X^c_t(x),X^c_{h_{cs}(t)}(y_{cs}))<C\epsilon_0$ for every $t\in \R$, where $d_{cs}$ denotes the intrinsic distance inside leaves of $\W^{cs}$ and $X^c_{h_{cs}(t)}(y_{cs})$ is  the intersection of $\W^u_{C\epsilon_0}(X^c_{h(t)}(y))$ and $\W^{cs}_{C\epsilon_0}(X^c_t(x))$ for every $t\in \R$.

Since $X^c_t$ is $C\epsilon_0$-expansive inside leaves of $\W^{cs}$ it follows that $y_{cs}$ must lie in a local piece of $X^c_t$-orbit of $x$. Analogously, using that $X^c_t$ is $C\epsilon_0$-expansive inside leaves of $\W^{cu}$ it follows that the point $y_{cu}:=\W^s_{C\epsilon_0}(y)\cap \W^{cu}_{C\epsilon_0}(x)$ lies also in a local piece of $X^c_t$-orbit of $x$. As a consequence of both facts, $y$ itself must lie in a local piece of $X^c_t$-orbit of $x$. This shows the claim.

Let us see that $X^c_t$ is $C\epsilon_0$-expansive inside $\W^{cu}$ leaves. For $\W^{cs}$ leaves the reasoning is analogous. Consider $x\in M$,  $y\in \W^{cu}(x)$ and $h:\R\to \R$ an increasing homeomorphism such that $h(0)=0$ and $d_{cu}(X^c_t(x),X^c_{h(t)}(y))<C\epsilon_0$ for every $t\in \R$. Let $y_u$ denote the intersection of $\W^c_{C^2\epsilon_0}(y)$ with $\W^u_{C^2\epsilon_0}(x)$. There exists $h_u:\R\to \R$ an increasing homeomorphism satisfying that $h_u(0)=0$ and $X^c_{h_u(t)}(y_u)$ equal to the intersection of $\W^c_{C^2\epsilon_0}(X^c_{h(t)}(y))$ with $\W^u_{C^2\epsilon_0}(X^c_t(x))$ for every $t\in \R$. If we show that $y_u=x$ then we show that $y$ lies in a local piece of the $X^c_t$-orbit of $x$.

For simplicity in the notation, let us rename $y=y_u$ and $h=h_u$. Suppose by contradiction that $x\neq y$. As a consequence, the point $X^c_{h(t)}(y)$ lies in $\W^u_{C^2\epsilon_0}(X^c_t(x))\setminus X^c_t(x)$ for every $t\in \R$. 

Let $N>0$ be such that $f^{-N}$ contracts distances inside $\W^u$ leaves. For every $n\geq N$ there exists $t_n\in \R$ so that $X^c_{t_n}(x)=f^n(x)$. One obtains that $z_n:=f^{-n}(X^c_{h(t_n)}(y))$ is a sequence in $\W^c(y)\cap \W^u_{C^2\epsilon_0}(x)\setminus x$ converging to $x$. Analogously, there exists $t_n'\in \R$ so that $f^n(y)=X^c_{h(t_n')}(y)$ for every $n\geq N$ and then $w_n:=f^{-n}(X^c_{t_n'}(x))$ is a sequence in $\W^c(x)\cap \W^u_{C^2\epsilon_0}(y)\setminus y$ converging to $y$.

Let $\delta_n>0$ be such that $\W^u_{\delta_n}(z_n)$ is a subset of $\W^u_{C^2\epsilon_0}(x)\setminus x$ for every $n\geq N$. Consider a compact center segment $[y,z_n]_c$ in $\W^c(y)$ joining $y$ with $z_n$.  For each $n$, by taking $w_{k_n}$ close enough to $y$ one can construct a center segment $[w_{k_n},w_n']_c$ close enough to $[y,z_n]_c$ so that $w_n'$ lies in $\W^u_{\delta_n}(z_n)$. 

Since $z_n$ converges to $x$ then $w_n'$ converges to $x$ as well. Moreover, $w_n'$ lies in $\W^c(x)\cap \W^u_{C^2\epsilon_0}(x)$ for every $n$ and is different from $x$ since $\W^u_{\delta_n}(z_n)$ is disjoint from $\{x\}$. One obtains that $x$ is a $u$-recurrent point. By Proposition \ref{propnouautobads2} this gives us a contradiction.
\end{proof}

Recall that $Y\subset M$ is the set of continuity points for $\rho$ whose center leaf is a line.

\begin{lemma}\label{lemmanofixedpointsinY}
The are no fixed points in $Y$.
\end{lemma}
\begin{proof}
Suppose by contradiction that there exists $x\in Y$ such that $f(x)=x$. Let $U$ be a foliation box neighborhood for $\W^c$ containing $x$. Since $\rho(x)=0$ and $\rho$ is continuous at $x$ there exists $\epsilon>0$ such that for every $y$ satisfying $d(x,y)<\epsilon$ one has that $y$ and $f(y)$ lie in the same center plaque of $U$. This contradicts the fact that there exists $y$ in $\W^s_\epsilon(x)$ such that $d_s(x,f^n(y))<\epsilon$ for every $n\geq 0$ and $\lim_n f^n(y)=x$.
\end{proof}

As a consequence of the above lemma the set $Y$ can be decomposed as the disjoint union of the following two sets, $Y^+$ and $Y^-$.

\begin{definition}
We define $$Y^+=\{x\in Y \mid f(x)\text{ and } X^c_1(x) \text{ lie in the same c.c. of } \W^c(x)\setminus x\}$$ and $$Y^-=\{x\in Y \mid f(x)\text{ and } X^c_1(x) \text{ lie in a different c.c. of } \W^c(x)\setminus x\}.$$
\end{definition}

In the above definition `c.c.' stands for `connected component'.

\begin{definition}
For every $x\in M$ let us denote $\alpha^c(x)$ and $\omega^c(x)$ the alpha and omega limit of the point $x$ for the center flow $X^c_t$, respectively. That is, $\alpha^c(x)=\{y\in M \mid \exists \text{ }(t_n)_n \text{ so that } \lim_n t_n=-\infty \text{ and } \lim_n X^c_{t_n}(x)=y\}$. Analogously for $\omega^c(x)$ when $\lim_n t_n=+\infty$.
\end{definition}

The next proposition allows us to extend $s$ and $u$ continuity of $\rho$ beyond the set $Y$.

\begin{prop}\label{propalphaomega_c}

For every $x^+\in Y^+$ and $x^-\in Y^-$ one has that:
\begin{itemize}
\item Every point in $\alpha^c(x^+)\cup\omega^c(x^-)$ is a $s$-continuity point for $\rho$.
\item Every point in $\alpha^c(x^-)\cup\omega^c(x^+)$ is a $u$-continuity point for $\rho$.
\end{itemize} 
\end{prop}
\begin{proof}
Suppose $x^+\in Y^+$ and $y\in \alpha^c(x^+)$. Let us see that $y$ is a $s$-continuity point for $\rho$. The other cases are analogous. For simplicity, let us denote $x^+$ by $x$.

By Proposition \ref{propcorofnoautobads} it is enough to show that for every $\epsilon>0$ there exists $\delta>0$ such that $\W^s_\delta(y)\subset \Stable^+_\epsilon(y)$. Without loss of generality, it is enough to show this for every $\epsilon$ smaller than $\epsilon_0$ the constant from Proposition \ref{propexpansiveflow}.

Suppose from now on $\epsilon>0$ smaller than $\epsilon_0$. The proof is carried out in a series of claims.

\begin{claim1}
There exists $\delta'>0$ so that for every $z\in (-\infty,x]_c$ and $w\in \W^s_{\delta'}(z)$ there exists a homeomorphism $h_w:[0,+\infty)\to [0,+\infty)$ satisfying that $X^c_{h_w(t)}(w)$ lies in $\W^s(X^c_t(z))$ for every $t\geq 0$,
\begin{equation}\label{eq2}
\lim_t \sup_{w\in \W^s_{\delta'}(z)} d_s(X^c_t(z),X^c_{h_w(t)}(w))=0
\end{equation}
and
\begin{equation}\label{eq1}
 \sup_{w\in \W^s_{\delta'}(z)} d_s(X^c_t(z),X^c_{h_w(t)}(w))\leq \epsilon 
\end{equation}
for every $t\geq T_z$, where $T_z\geq 0$ denotes the time such that $X^c_{T_z}(z)=x$. 
\end{claim1}

\begin{proof}\let\qed\relax To show this claim, let us look more closely at the behavior of center leaves at $s$-continuity points. Since $x$ is a $s$-continuity point for $\rho$, there exists $r>0$ such that $\rho$ is continuous in restriction to the stable disc $\W^s_r(x)$. Since $\W^c(x)$ is a line, by Lemma \ref{lemmatwocircleleaves} one can consider $r>0$ small enough so that every center leaf through $\W^s_r(x)$ is also a line.

The same argument for showing that $Y$ is saturated by center leaves in the proof of Proposition \ref{propYresidualWcsatpathconn} shows that every point in a center leaf through $\W^s_r(x)$ is a $s$-continuity point for $\rho$. And the same argument for showing Lemma \ref{lemmanofixedpointsinY} shows there are no fixed points of $f$ in any center leaf through $\W^s_r(x)$. In particular, for every $w\in \W^s_r(x)$ the half-open center segment $[w,f(w))_c$ is a fundamental domain for the dynamics of $f$ restricted $\W^c(w)$.

The continuity of $\rho$ in restriction to $\W^s_r(x)$ means that the center segments $[w,f(w)]_c$ vary continuously in the Hausdorff topology as $w$ varies in $\W^s_r(x)$. One can consider $r>0$ small enough so that $\bigcup_{w\in \W^s_{r}(x)}[w,f(w)]_c$ is a foliation box neighborhood of the foliation $\W^c$. 

Moreover, there exist $r''> r'>0$ so that $\bigcup_{w\in \W^s_{r}(x)}[w,f(w)]_c$ contains the set $\W^s_{r'}([x,f(x)]_c)$ and is contained in $\W^s_{r''}([x,f(x)]_c)$. Modulo taking $r>0$ smaller, the constant $r''>0$ can be considered small enough so that $\W^s_{r''}(z)$ and $\W^s_{r''}(z')$ are disjoint discs for every pair $z,z'\in [x,f(x)]_c$ such that $z\neq z'$.

It is immediate that in $\bigcup_{w\in \W^s_{r}(x)}[w,f(w))_c$, if $w$ is a point in  $\W^s_r(z)$, then there exists a unique curve $\gamma_w$ so that $\gamma_w(t)=\W^s_{r''}(X^c_t(x))\cap [w,f(w))_c$ for every $t$ such that $X^c_t(x)$ lies in $[x,f(x))_c$. This can be extended naturally to every $t\in\R$ as follows. If $X^c_t(x)$ lies in $[f^n(x),f^{n+1}(x))_c$ for some $n\in \mathbb{Z}$, and $x_t$ denotes the point in $[x,f(x))_c$ such that $f^n(x_t)=X^c_t(x)$, then $\gamma_w(t)$ is equal to the intersection of $f^n(\W^s_{r''}(x_t))$ and $[f^n(w),f^{n+1}(w))_c$. The fact that $\rho$ is continuous in restriction to $\W^s_r(x)$ translates to the fact that for every $w\in \W^s_r(x)$ the curve $\gamma_w$ is continuous for every $t\in \R$.

From the above, it follows that for every $T\in \R$ and $w\in \W^s_r(x)$ the points $z:=X^c_T(x)$ and $w':=\gamma_w(T)$ satisfy that there exists a unique homeomorphism $h_{w'}:[0,+\infty)\to [0,+\infty)$ such that $X^c_{h_{w'}(t)}(w')$ lies in $\W^s(z)$ for every $t\in \R$. Namely, $h_{w'}(t)$ is such that $X^c_{h_{w'}(t)}(w')=\gamma_w(T+t)$ for every $t\in \R$.

Again, modulo taking $r>0$ smaller beforehand, let us consider $r',r''>0 $ small enough so that $f^n(\W^s_{r''}(p))$ is included in $\W^s_\epsilon(f^n(p))$ for every $p\in M$ and $n\geq 0$. And such that for some constant $\delta'>0$ one has that $\W^s_{\delta'}(f^{-n}(p))$ is contained in $f^{-n}(\W^s_{r'}(p))$ for every $p\in M$ and $n\geq 0$. It follows that (\ref{eq2}) and  (\ref{eq1}) are immediately satisfied for every $z\in (-\infty,x]_c$ and $w\in \W^s_{\delta'}(z)$. This proves the claim.
\end{proof}

Note that, by the continuity of $X^c_t$, for every $z\in (-\infty,x]_c$ there exists $\eta>0$ smaller or equal to $\delta'$ so that, if one considers $\W^s_\eta(z)$ in the place of $\W^s_{\delta'}(z)$, then (\ref{eq1}) is satisfied for every $t\geq 0$ instead of for every $t\geq T_z$. Let $\eta_z>0$ denote the supremum of the constants $\eta>0$ satisfying this. It immediate that $\W^s_{\eta_z}(z)$ is a subset of $\Stable^+_\epsilon(z)$. Note also that, by definition, $\eta_z\leq \delta'$. 

\begin{claim2}
If $\eta_z<\delta'$ for some $z\in (-\infty,x]_c$, then for some $w\in \W^s_{\delta'}(z)$ satisfying  $d_s(z,w)=\eta_z$ one has that 
\begin{equation}\label{eq3} d_s(X^c_t(z),X^c_{h_{w}(t)})\leq \epsilon
\end{equation}
for every $t\geq 0$ and that 
\begin{equation}\label{eq4}
d_s(X^c_{T}(z),X^c_{h_{w}(T)})= \epsilon
\end{equation} for some $T\geq 0$.
\end{claim2}

\begin{proof}\let\qed\relax Indeed, on the one hand, if it were the case that for some $w\in \W^s_{\delta'}(z)$ so that $d_s(z,w)=\eta_z$ one has that (\ref{eq3}) is not satisfied for every $t\geq 0$ then by the continuity of $X^c_t$ one would have that the same happens to every $w$ in $\W^s_{\delta'}(z)$ close enough to $w$, contradicting that $\eta_z$ is a supremum since a lower upper bound would exist. This shows that inequality (\ref{eq3}) needs to be satisfied for every $w\in \W^s_{\delta'}(z)$ such that $d_s(z,w)=\eta_z$.

On the other hand, by (\ref{eq2}) there exists $T'\geq 0$ so that for every $w\in \W^s_{\delta'}(z)$ and $t\geq T'$ one has that $d_s(X^c_t(z),X^c_{h_w(t)}(w))<\epsilon/2$. In case there exists $w\in \W^s_{\delta'}(z)$ with $d_s(z,w)=\eta_z$ and satisfying a strict inequality in (\ref{eq3}) for every $t\in [0,T']$, then one obtains by the continuity of $X^c_t$ that a strict inequality in (\ref{eq3}) happens for every $w'$ in a neighborhood of $w$ in $\W^s_{\delta'}(z)$ for every $t\geq 0$. This shows that there exists some $w$ in $\W^s_{\delta'}(z)$ with $d_s(z,w)=\eta_z$ satisfying that $d_s(X^c_T(z),X^c_{h_{w}(T)})= \epsilon$ for some $T\geq 0$, otherwise $\eta_z$ would not be a supremum since $\eta_z$ would not be an upper bound. This proves the claim.
\end{proof}

Since $\W^s_{\eta_z}(z)\subset\Stable^+_\epsilon(z)$ for every $z\in(-\infty,x]_c$, the next claim gives us uniform size of forwards $\epsilon$-stable set of $X^c_t$ at the stable leaf of every point in $(-\infty,x]_c$.

\begin{claim3}
There exists $\delta>0$ such that $\eta_z>\delta$ for every $z\in (-\infty,x]_c$.
\end{claim3}

\begin{proof}\let\qed\relax Suppose by contradiction that there exists a sequence $(z_n)_n$ in $(-\infty,x]_c$ so that $(\eta_{z_n})_n$ tends to $0$. Without loss of generality, suppose $\eta_{z_n}<\delta'$ for every $n$. By the previous claim, one has that for every $n$ there exists $w_n\in \W^s_{\delta'}(z_n)$ such that $d_s(z_n,w_n)=\eta_{z_n}$ and $$d_s(X^c_{T_n}(z_n),X^c_{h_{w_n}(T_n)}(w_n))= \epsilon$$ for some $T_n\geq 0$. Note that, by the continuity of $X^c_t$, since $(\eta_{z_n})_n$ tends to $0$ then $(T_n)_n$ tends to $+\infty$ .

For every $n$ let us denote $p_n=X^c_{T_n}(z_n)$ and $q_n=X^c_{h_{w_n}(T_n)}$. One has that $d_s(p_n,q_n)=\epsilon$ and that there exists a homeomorphism $h_n:[-T_n,+\infty)\to[-T_n,+\infty)$ such that $$d_s(X^c_t(p_n),X^c_{h_n(t)}(q_n))\leq \epsilon$$ for every $t\in [-T_n,+\infty)$.

Let $(p,q)\in M\times M$ be an accumulation point of the sequence $((p_n,q_n))_n$. One obtains that $d_s(p,q)=\epsilon$ and, since  $(T_n)_n$ tends to $+\infty$, that there exists an orientation preserving homeomorphism $h_q:\R\to \R$ with $h_q(0)=0$ so that $$d_s(X^c_t(p),X^c_{h_q(t)}(q))\leq \epsilon$$ for every $t\in \R$. This contradicts Proposition \ref{propexpansiveflow}. This proves the claim.
\end{proof}

We have shown that there exists $\delta>0$ such that for every $z\in (-\infty,x]_c$ the set $\W^s_{\delta}(z)$ is a subset of $\Stable^+_\epsilon(z)$. It is immediate, by the continuity of $X^c_t$, that this passes to the limit and one obtains that $\W^s_{\delta}(y)$ is a subset of $\Stable^+_\epsilon(y)$. This ends the proof of the proposition.
\end{proof}

The proof of the next corollary is the only place where the hypothesis of a dense leaf of $\W^c$  is needed.

\begin{cor}\label{corestrellatransitive}
The function $\rho$ is continuous at every $x$ in $M$ such that $\W^c(x)$ is a line. 
\end{cor}

\begin{proof} Since $X^c_t$ is transitive the set $\{x\in M \mid \alpha^c(x)=\omega^c(x)=M\}$ is a residual subset of $M$. The set $Y$ is also residual. Thus there exists  $x_0$ in $Y$ such that $\alpha^c(x_0)=\omega^c(x_0)=M$. By Proposition \ref{propalphaomega_c} every point $z \in M$ such that $\W^c(z)$ is a line is a $s$ and $u$-continuity point for $\rho$. By Lemma \ref{lemmasanducontpointtheninX} one concludes.
\end{proof}

For completeness, let us state what we know how without assuming a dense leaf of $\W^c$:

\begin{remark}\label{rmkhypWctransitive}
The same arguments for showing Proposition \ref{propalphaomega_c} allows one to conclude that if $y$ is a point in $\alpha^c(x^+)$ or $\omega^c(x^-)$ for some $x^+\in Y^+$ or $x^-\in Y^-$, then every point $z$ in $\W^s(y)$ such that $\W^c(z)$ is a line is a $s$-continuity point for $\rho$. And the analogous statement for $u$-continuity points.

Moreover, it is not difficult to show that if $x$ is a point that does not lie in the non-wandering set $\Omega(X^c_t)$ of the center flow $X^c_t$, then $x$ is a continuity point of $\rho$.

If it were the case that one is able to deduce from the above that $\rho$ needs to be continuous at every line leaf of $\W^c$, then one would be able to conclude that $f$ has to be a discretized Anosov flow because the arguments in the next section do not need the hypothesis of a dense center leaf.
\end{remark}

\section{The function $\rho$ is bounded in $M$}\label{sectionrhobounded}

This section ends the proof of Theorem \ref{thm1}. Up until now we have shown that the function $\rho$ is continuous at every $x$ in $M$ such that $\W^c(x)$ is a line.  We will see in this section how to show from this that $\rho$ is bounded in $M$.

Given $\C\in \W^c$ a circle leaf, it follows from Lemma \ref{lemmatwocircleleaves} that $\W^c(y)$ is a line for every $y$ in $\W^s(\C)\setminus \C$. In particular, the center segment $[y,f(y)]_c$ is well defined for every $y\in \W^s(\C)\setminus \C$ and by Corollary \ref{corestrellatransitive} the function $\rho$ is continuous in restriction to $\W^s(\C)\setminus \C$.

Recall that the stable saturation $\W^s(\C)$ of $\C$ is contained in the center-stable leaf $\W^{cs}(\C)$ but that a priori $\W^s(\C)$ may be a proper subset of $\W^{cs}(\C)$ (this in known as the \emph{completeness} problem).

\begin{lemma}\label{lemma[y,f(y)]_csubsetW^s(C)}
Suppose $\C\in \W^c$ is a circle leaf. For every $y\in \W^s(\C)\setminus \C$ one has that $[y,f(y)]_c$ is contained in $\W^s(\C)$.
\end{lemma}
\begin{proof}
Since $\rho$ is continuous in restriction to $\W^s(\C)\setminus \C$ one has that for every $y\in \W^s(\C)\setminus \C$ the compact center segment $[y,f(y)]_c$ varies continuously with $y$ in the Hausdorff topology.

Let $A$ denote the set of points $y\in \W^s(\C)\setminus \C$ such that $[y,f(y)]_c$ is contained in $\W^s(\C)\setminus \C$. Let $B$ denote its complement in $\W^s(\C)\setminus \C$ so that $\W^s(\C)\setminus \C$ is equal to the disjoint union $A\cup B$. The goal is to show that $B$ is empty.

Let $\partial_{cs}\W^s(\C)\subset \W^{cs}(\C)$  denote the boundary of $\W^s(\C)$ in $\W^{cs}(\C)$. Since $\W^s(\C)$ is saturated by leaves of $\W^s$ the set $\partial_{cs}\W^s(\C)$ is a union of leaves of $\W^s$. It follows that $y\in \W^s(\C)\setminus \C$ is in $B$ if and only if $[y,f(y)]_c\cap \partial_{cs}\W^s(\C)\neq \emptyset$. 

Since $y\mapsto [y,f(y)]_c$ varies continuously with $y$ in $\W^s(\C)\setminus \C$ it is immediate to check that both $A$ and $B$ are open subsets of $\W^s(\C)\setminus\C$ (for this, note that if $y\in B$ then $[y,f(y)]_c$ is transverse to $\W^s(f(y))$ at $f(y)$). 

As $\W^s(\C)\setminus \C$ is the union of the disjoint open sets $A$ and $B$ it follows that $A$ and $B$ comprise whole connected components of $\W^s(\C)\setminus \C$. Note that if $\dim(E^s)\geq 2$ then $\W^s(\C)\setminus \C$ has only one connected component and if $\dim(E^s)=1$ it may have two. We will cover both scenarios simultaneously. 

Suppose by contradiction that $B$ is not empty. For every $y\in B$ the center segment $[y,f(y)]_c$ intersects $\partial_{cs}\W^s(\C)$. Since $\W^s(\C)$ is $f$-invariant it follows that $f^{-1}\circ [y,f(y)]_c=[f^{-1}(y),y]_c$ also intersects $\partial_{cs}\W^s(\C)$. One can then consider $y^+$ and $y^-$ the `first time' that $\W^c(y)$ leaves $\W^s(\C)$ in both directions. That is, $y^-$ and $y^+$ are the only points in $\W^c(y)\cap \partial_{cs}\W^s(\C)$ such that there exists a center segment $(y^-,y^+)_c$ contained in $\W^s(\C)$ and satisfying $y \in (y^-,y^+)_c$.  In other words, $(y^-,y^+)_c$ is the connected component of $\W^c(y)\cap \W^s(\C)$ containing $y$.

It is immediate to check (by transversality again) that the functions $y\mapsto y^+$ and $y\mapsto y^-$ are continuous from $B$ to $\partial_{cs}\W^s(\C)$. Moreover, if $z$ is a point in $\partial_{cs}\W^s(\C)$ that is in the image of $y\mapsto y^+$, then every $z'\in \W^s(z)$ need to be also in the image of $y\mapsto y^+$. This is because, by stable holonomy, one can transport $(y,y^+)_c$ to a center segment in $\W^s(\C)$ such that one of its endpoints is $z'$.

Moreover, for every $y\in B$ it is immediate to check that there exists $\epsilon>0$ such that $w^+$ lies in $\W^s(y^+)$ for every $w\in \W^{cs}_\epsilon(y)$. That is, the function that assigns to every $y\in B$ the stable leaf $\W^s(y^+)$ is locally constant. Combined with the information from the previous paragraph one obtains that the image by $y\mapsto y^+$ of $B$ is exactly one or two leaves of $\W^s$, whether $B$ has one or two connected components, respectively.

Let $V$ be one of the leaves of $\W^s$ in the image of $y\mapsto y^+$. Since $B$ is $f$-invariant and has at most two connected components then $V$ is invariant by $f^2$. It follows that $f^2$ induces a contraction in $V$. As a consequence $f^2$ has a fixed point in $V$ and this fixed point is unique. Let $y_0\in B$ be such that $y_0^+$ is the fixed point of $f^2$ in $V$. 

On the one hand, since $\W^s(\C)$ and $\partial_{cs}\W^s(\C)$ are $f$-invariant then the image of $[y_0^-,y_0^+]_c$ by $f^2$ is a center segment whose interior lies in $\W^s(\C)$  and its end-points lie in $\partial_{cs}\W^s(\C)$. Since $y_0^+$ is fixed by $f^2$ it follows that $[y_0^-,y_0^+]_c$ is invariant by $f^2$.

On the other hand, $[y_0^-,y_0^+]_c$ contains the point $y_0$ which is a point in $\W^s(x_0)$ for some $x_0\in\C$. By iterating forwards by $f^2$ one obtains that the orbit of $y_0$ needs to get arbitrarily close to $\C$. Since $[y_0^-,y_0^+]_c$ is $f^2$-invariant this contradicts the fact that $[y_0^-,y_0^+]_c$ and $\C$ are disjoint compact sets that are at a positive distance from each other. This shows that the set $B$ needs to be empty and ends the proof of the lemma.
\end{proof}

\begin{lemma}\label{lemmarhonoundedinWsloc}
Suppose $\C\in \W^c$ is a circle leaf. There exists $\delta>0$ such that $\rho$ restricted to $\W^s_\delta(\C)$ is bounded.
\end{lemma}
\begin{proof} The set $\W^s(\C)\setminus \C$ has one or two connected components. Without loss of generality let us suppose that is has one. Otherwise, one should only repeat the argument below separately on each connected component.

Suppose $y$ in $\W^s(\C)\setminus \C$. By Lemma \ref{lemmatwocircleleaves} the leaf $\W^s(y)$ intersects $\C$ in a unique point. Let us call it $p^sy$.

Let $\gamma_y:[0,1]\to \W^c(y)$ be the $C^1$ curve of constant speed such that $\gamma_y(0)=y$ and $\gamma_y(1)=f(y)$. By Lemma \ref{lemma[y,f(y)]_csubsetW^s(C)} the center segment $[y,f(y)]_c$ is contained in $\W^s(\C)$.  It follows that there exists $p^s\gamma_y:[0,1]
\to \C$ the (unique) continuous curve such that $p^s\gamma_y(0)=p^sy$ and $
\gamma_y(t) \in \W^s(p^s\gamma_y(t))$ for every $t\in [0,1]$. 

By Corollary \ref{corestrellatransitive} the function $\rho$ is continuous in restriction to $\W^s(\C)\setminus \C$. That is, $
\gamma_y$ varies continuously with $y\in \W^s(\C)\setminus \C$ in the $C^1$ topology. At the same time, if $y'$ varies continuously in $\W^s(p^sy)\setminus p^sy$ one has that $f(y')$ varies continuously in $\W^s(f(p^sy))\setminus f(p^sy)$. One obtains that $p^s\gamma_y$ needs to be a reparametrization of $p^s\gamma_{y'}$ for every $y'$ in $\W^s(y)$.

Given $x$ in $\C$ and $y\in \W^s(x)\setminus x$ let $\gamma_x:[0,1]\to \C$ be the constant speed reparametrization of $p^s\gamma_y$. From the above paragraph one has that the definition of $\gamma_x$ is independent of the point $y$ one chooses in $\W^s(x)\setminus x$.

It is now immediate to check that $z\mapsto \gamma_z$ varies continuously in the $C^1$ topology as $z$ varies in $\W^s(\C)$. Since this implies that $z\mapsto \length(\gamma_z)$ varies continuously with $z\in \W^s(\C)$ it follows that $\rho$ is continuous in a neighborhood of $\C$ in $\W^s(\C)$.
\end{proof}

\begin{prop}\label{proprhoboundednearcircleleaf}
For every circle leaf $\C\in \W^c$ there exists a neighborhood $U$ of $\C$ such that $\rho|_U:U\to \R$ is bounded.
\end{prop}
\begin{proof}
Let  $\C\in \W^c$ be a circle leaf. It follows from Lemma \ref{lemmarhonoundedinWsloc} that there exists $\delta>0$ and $L>0$ such that $\rho$ restricted to $\W^s_\delta(\C)$ is bounded by $L$.

From the regularity of $\W^c$ one can define an unstable holonomy along center transversals as follows: There exists $\delta_L>0$ such that, if $y$ is a point in $M$ and $z$ a point in  $\W^u_{\delta_L}(y)$, then for every curve $\gamma:[0,1]\to \W^c(y)$  such that $\gamma(0)=y$ and $\length \gamma \leq L$ there exists a unique curve $p^u\gamma :[0,1]\to \W^c(z)$ given by $p^u \gamma (0)=z$ and $p^u\gamma(t)\in \W^u(\gamma(t))$ for every $t\in [0,1]$, and this curve satisfies that $\length p^u \gamma \leq 2L$.

Let us see that $\rho$ is bounded by $2L$ in $U=\W^u_{\delta_L}(\W^s_\delta(\C)))$. Since $U$ is a neighborhood of $\C$ this will show the proposition.

Given $z$ in $U$ there exists $x\in \C$ and $y\in \W^s_\delta(x)$ such that $z\in \W^u_{\delta_L}(y)$. Let us suppose first that for every $w$ in $\W^u_{\delta_L}(y)$ the center leaf $\W^c(w)$ is not compact. We can join then $y$ with $z$ by a curve $\eta:[0,1]\to \W^u_{\delta_L}(y)$ satisfying that $\eta(0)=y$, $\eta(1)=z$ and $\W^c(\eta(s))$ is a line for every $s\in [0,1]$.

By Corollary \ref{corestrellatransitive} one has that $\rho$ is continuous at every point in the image of $\eta$. Let $\gamma:[0,1]\to \W^c(y)$ be a homeomorphism from $[0,1]$ to $[y,f(y)]_c$ such that $\gamma(0)=y$ and $\gamma(1)=f(y)$. Since $\rho$ is continuous in the image of $\eta$ it follows that for every $s\in [0,1]$ there exists $\gamma_s:[0,1]\to \W^c(\eta(s))$ joining $\eta(s)=\gamma_s(0)$ and $f(\eta(s))=\gamma_s(1)$, and satisfying that $\gamma_0=\gamma$ and  $\gamma_s(t)\in \W^u(\gamma(t))$ for every $t\in [0,1]$. In particular, $f(z)=\gamma_1(1)$.

One has that $\length \gamma \leq L$ since $\rho$ is bounded by $L$ in $\W^s_\delta(\C)$. Then by the election of $\delta_L$ it follows that $\length \gamma_1 \leq 2L$. Since $\gamma_1$ is a curve in $\W^c(z)$ joining $z=\gamma_1(0)$ with $f(z)=\gamma_1(1)$ one obtains that $\rho(z)\leq 2L$.

In case $\W^u_{\delta_L}(y)$ intersects a compact leaf of $\W^c$ one can argue as follows. By Lemma \ref{lemmafinitelymanyleavesoflength<R} all but countably many $y'\in \W^s_\delta(x)$ satisfy that $\W^u_{\delta_L}(y')$ does not intersect a compact leaf of $\W^c$. One can consider then $(y_n)$ a sequence in $\W^s_\delta(x)$ converging to $y$ such that for every $w\in \W^u_{\delta_L}(y_n)$ the center leaf $\W^c(w)$ is not compact. And consider $z_n \in \W^u_{\delta_L}(y_n)$, for every $n$, so that the sequence $(z_n)$ converges to $z$.

By the arguments above one has that $\rho(z_n)\leq 2L$ for every $n$. By the semicontinuity of $\rho$ (see Proposition \ref{proprhosemicont}) it follows that $\rho(z)\leq 2L$. 
\end{proof}

As pointed out in Remark \ref{rmkDAFiftaubounded}, the following ends the proof of Theorem \ref{thm1} as a consequence of Proposition \ref{propcenterfixingL}.

\begin{cor}\label{corfinal}
The function $\rho$ is bounded in $M$
\end{cor}
\begin{proof}
Suppose $x\in M$. If $\W^c(x)$ is a line then by Corollary 
\ref{corestrellatransitive} the function $\rho$ is continuous at $x$. In particular, it is bounded in a neighborhood of $x$. If $\W^c(x)$ is a circle, then by Proposition \ref{proprhoboundednearcircleleaf} the function $\rho$ is bounded on a neighborhood of $\W^c(x)$. By compactness of $M$ one obtains that $\rho$ is bounded in $M$.
\end{proof}

\section{Homeomorphisms `fixing' a one-dimensional foliation}\label{sectionhomeos}

One might wonder whether the statement of Theorem \ref{thm1} is not obvious and true in general for continuous maps `fixing' the leaves of a one-dimensional foliation. Let us frame this into a general problem.

\begin{problem}\label{problem1} Suppose $f:M\to M$ is a homeomorphism such that $f(W)=W$ for every leaf $W$ in a one-dimensional foliation $\W$. Under which circumstances does there exists a continuous flow $X_t$ whose orbits are the leaves of $\W$ and a continuous function $\tau:M\to \mathbb{R}$ such that $f(x)=X_{\tau(x)}(x)$ for every $x\in M$?
\end{problem}

Of course, a necessary condition is that $\W$ be orientable. Other conditions (on $f$ or $\W$) are certainly needed. We will present a few examples to illustrate this.

\begin{example}\label{ex1}
Consider $\psi:S^1\to S^1$ a North-South map in $S^1=\mathbb{R}/\mathbb{Z}$ and $X_t$ its suspension flow, generated by the constant vector field $X=(0,1)$ in the suspension manifold $M:=S^1\times [0,1]/_\sim$ given by $(x,1)\sim (\psi(x),0)$ for every $x\in S^1$.
It follows that the map $f:M\to M$ given by 
\[ 
f(\theta, s)=
\begin{cases}
          (\psi(r,s) & \text{ if } 0\leq r< 1/2\\
		  (\psi^2(r,s) & \text{ if } 1/2 \leq r< 1
\end{cases}
\]
is a homeomorphism such that $f(W)=W$ for every leaf $W$ in the foliation $\W$ given by the orbits of $X_t$. However, it is immediate that there is no continuous function $\tau:M\to\R$ such that $f(x)=X_{\tau(x)}(x)$ for every $x\in M$. See Figure \ref{fig3}.
\end{example}

\begin{figure}
\centering
\begin{minipage}{.50\textwidth}
  \centering
  {\scriptsize
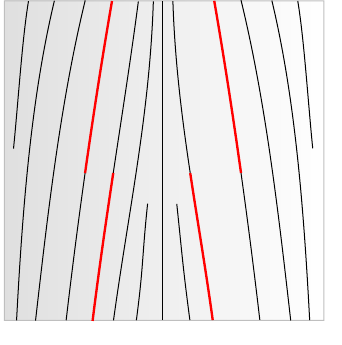
}
  \caption{At $r=\frac{1}{2}$ the function $x\mapsto d_\W(x,f(x))$ tends to 1 from the left and to 2 from the right.}
  \label{fig3}
\end{minipage}%
\begin{minipage}{.51\textwidth}
  \centering
  {\scriptsize
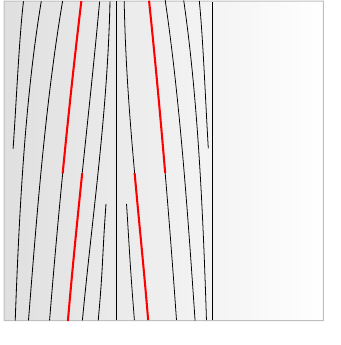
}
  \caption{The function $x\mapsto d_\W(x,f(x))$ is locally unbounded at $r=0$ on each annulus $\theta=const.$}
  \label{fig4}
\end{minipage}
\end{figure}

The above example is a homeomorphism in the 2-torus `fixing' a foliation $\W$  so that the function $x\mapsto d_\W(x,f(x))$ is not continuous. The next one will make a similar construction repeatedly to obtain an example where the function $x\mapsto d_\W(x,f(x))$ is unbounded. 

\begin{example}\label{ex2} 
Consider $\varphi:[0,1]\to [0,1]$ a homeomorphism whose fixed points are exactly $\{0\}\cup\{1/n\mid n\in \mathbb{Z}^+\}$. In the closed 2-disc $D^2$, in polar coordinates $(r,\theta)$ so that $r\in [0,1]$ and $\theta\in S^1$,
 consider the map $\psi(r,\theta)=(\varphi(r),\theta)$. Let $X_t$ be the suspension flow of $\psi$ in the suspension manifold $M:=D^2\times [0,1]/_\sim$ given by $(x,1)\sim (\psi(x),0)$ for every $x\in D^2$. 
 
Consider the map $f:M\to M$ given by $$f(r,\theta,s)=(\varphi^n(r),\theta,s)$$ if $1/(n+1)< r \leq 1/n$ for $n\in \mathbb{Z}^+$ and the identity if $r=0$. It follows that $f$ is a homeomorphism such that $f(W)=W$ for every leaf $W$ in the foliation $\W$ given by the orbits of $X_t$. It is immediate that there is no continuous function $\tau:M\to\R$ such that $f(x)=X_{\tau(x)}(x)$ for every $x\in M$. Moreover, the function $x\mapsto d_\W(x,f(x))$ is locally unbounded at every point in the circle $r=0$. See Figure \ref{fig4}. 

By considering two copies of this example glued adequately by their boundary (so that both foliations match to a foliation) one can obtain a foliation fixing homeomorphism $f$ in $M=S^2\times S^1$ so that $x\mapsto d_\W(x,f(x))$ is locally unbounded at every point in the union of two different circle leaves of $\W$ (one for each solid torus).
\end{example}

\begin{remark}
In the above examples the discontinuity issue for the function $x\mapsto d_\W(x,f(x))$ takes place only at compact leaves of $\W$. However, one can use the model from Example \ref{ex2} to construct a foliation fixing homeomorphism in dimension 3 so that the discontinuity problem for $x\mapsto d_\W(x,f(x))$ occurs at non-compact leaves.

Indeed, taking the foliation given by the flow lines of a transitive Anosov flow as a starting point, it is enough to construct a new foliation $\W$ by blowing-up the periodic orbits  $\{\gamma_i\}_{i\in \mathbb{N}}$ of the flow to two-by-two disjoint solid tori $\{T_i\}_{i\in \mathbb{N}}$ subfoliated as in Example \ref{ex2}. Then $f$ can be taken to be the identity in $M\setminus \bigcup_{i\in \mathbb{N}} T_i$ and a map as in Example \ref{ex2} in each torus $T_i$.

Every point $y$ in $M\setminus \bigcup_{i\in \mathbb{N}} T_i$ lies in a line leaf of $\W$ and is accumulated by points $y_n$ from $\bigcup_{i\in \mathbb{N}} T_i$ so that $x\mapsto d_\W(x,f(x))$ is locally unbounded at $y_n$. It turns out that $x\mapsto d_\W(x,f(x))$ is also locally unbounded at $y$.

One can even consider a flow $X_t$ with the leaves of $\W$ as orbits and compose $f$ with the time $\tau$ of $X_t$ for some continuous function $\tau: M\to \R_{>0}$. The resulting homeomorphism will present the same discontinuity issues for $x\mapsto d_\W(x,f(x))$  and will not be the identity in $M\setminus \bigcup_{i\in \mathbb{N}} T_i$.
\end{remark}

Let us finish by pointing out a scenario where we know an answer to Problem \ref{problem1}. We thank P. Lessa for this proposition.

\begin{prop}
Suppose $f:M\to M$ is a homeomorphism in a compact manifold $M$ and $\W$ is a one-dimensional foliation such that $f(W)=W$ for every leaf $W\in \W$. If $\W$ is minimal and orientable then there exists $\tau:M\to \mathbb{R}$ continuous such that $f(x)=X_{\tau(x)}(x)$ for every $x\in M$, where $X_t:M\to M$ is a continuous flow whose orbits are the leaves of $\W$.
\end{prop}
\begin{proof}
Let $X_t:M\to M$ be a continuous flow with the leaves of $\W$ as orbits. Since $\W$ is minimal every leaf of $\W$ is a line. So there is an unambiguously defined function $\tau:M\to \mathbb{R}$ such that $f(x)=X_{\tau(x)}(x)$ for every $x\in M$. The goal is to show that $\tau$ is continuous.

Define $M_n=\{x\in M \text{ s.t. }|\tau(x)|\leq n\}$. Note that each $M_n$ is closed and that $M=\bigcup_n M_n$. By Baire's category theorem it follows that $M_N$ has non-empty interior for some $N>0$.

By the minimality of $\W$ there exists $c>0$ such that the segment $X_{[x,x+c]}(x)$ intersects $M_N$ for every $x\in M$. This is the key information that the minimality of $\W$ gives. 

Fix now $x\in M$ and consider the metric $d_\W$ in $\W(x)$ so that $d_\W(X_t(x), X_s(x))$ is equal to $|t-s|$ for every $t,s\in \mathbb{R}$. We know that $M_N\cap \W^c(x)$ is $c$-dense in $\W^c(x)$ for $d_{
\W}$, so there exist $y$ and $z$ in $\W(x)\cap M_N$ such that $x\in [y,z]_\W$ and $d_{\W}(y,z)\leq c$.

Since $d_\W(y,f(y))\leq N$ and $d_\W(z,f(z))\leq N$, by triangular inequality (two times) it follows that $d_\W(f(y),f(z))\leq 2N+c$. Note that $f(x)$ lies in $[f(y),f(z)]_\W$. Depending on which side of $\W(x)$ with respect to $x$ lies $f(x)$, it is immediate that $d_
\W(x,f(x))$ can be bounded by $d_\W(y,f(y))+d_\W(f(y),f(z))$ or by $d_\W(z,f(z))+d_\W(f(y),f(z))$. One obtains that $$d(x,f(x))\leq 2N+2c.$$ As $x\in M$ was arbitrary, we have shown that $M_{2N+2c}=M$. That is, $\tau$ is bounded by $2N+2c$. Continuity of $\tau$ is immediate from this since every leaf of $\W$ is a line.
\end{proof}

We do not know if a similar result as the above is true in general provided that $\W$ is a transitive foliation, or if a counterexample can be made.

\end{document}